\documentclass[12pt,a4paper]{amsart}

\usepackage{amssymb}
\usepackage[mathscr]{eucal}
\usepackage{graphicx}

\usepackage{amscd}
\newtheorem{theorem}{Theorem}[section]
\newtheorem{lemma}[theorem]{Lemma}
\newtheorem{prop}[theorem]{Proposition}
\newtheorem{cor}[theorem]{Corollary}

\theoremstyle{definition}
\newtheorem{definition}[theorem]{Definition}

\theoremstyle{remark}

\numberwithin{equation}{section}

\newcommand{\cl}{\mathcal}

\theoremstyle{plain}

\newcommand{\adef}{\begin{definition}}
\newcommand{\zdef}{\end{definition}}

\newcommand{\R}{\mathbb{R}}
\newcommand{\N}{\mathbb{N}}

\newcommand{\aproof}{\begin{proof}}
\newcommand{\zproof}{\end{proof}}

\begin{document}

\title{Derivation of vector-valued complex interpolation scales}


\author{Jes\'us M. F. Castillo}
\address{Instituto de Matem\'aticas Imuex, Universidad de Extremadura,
Avenida de Elvas s/n, 06011 Badajoz,  Espa\~{n}a}
\curraddr{}
\email{castillo@unex.es}

\thanks{The research of the first and third author was supported by Project MTM2016-76958-C2-1-P and Project IB16056 de la Junta de Extremadura.}

\thanks{The authors are indebted to M. Cwikel for several accurate comments that helped to improve the paper}
\author{Daniel Morales}

\address{Departamento de Matem\'aticas, Universidad de Extremadura,
Avenida de Elvas s/n, 06011 Badajoz, Spain}
\email{ddmmgg1993@gmail.com}

\author{Jes\'us Su\'arez de la Fuente}
\address{Instituto de Matem\'aticas Imuex, Universidad de Extremadura,
Avenida de Elvas s/n, 06011 Badajoz,  Espa\~{n}a.}
\email{jesus@unex.es}


\date{}

\dedicatory{}


\begin{abstract} We study complex interpolation scales obtained by vector valued amalgamation and the derivations they generate. We study their trivial and singular character and obtain examples showing that the hypotheses in the main theorems of [J.M.F. Castillo, V. Ferenczi and M. Gonz\'alez, \emph{Singular exact sequences generated by complex interpolation},  Trans. Amer. Math. Soc. 369 (2017) 4671--4708] are not necessary. \end{abstract}

\maketitle

\section{Introduction}
\label{sect:intro}

In this paper we study interpolation scales of vector valued functions, the derivations they induce and some of their basic properties, mainly nontriviality and  singularity. Our examples show, in particular, that the hypothesis of the main results in [J.M.F. Castillo, V. Ferenczi and M. Gonzalez, \emph{Singular exact sequences generated by complex interpolation},  Trans. Amer. Math. Soc. 369 (2017) 4671--4708] are not necessary. A special attention is payed to the derivations obtained by amalgamation or fragmentation, in the spirit of the Enflo, Lindenstrauss and Pisier construction. Indeed, the first twisted Hilbert space was obtained by Enflo, Lindenstrauss and Pisier \cite{elp}. It has the form $\ell_2(F_n)$ for a specific sequence $F_n$ of finite-dimensional Banach spaces. Even if it is not known whether the Enflo-Lindenstrauss-Pisier space can be obtained by derivation, we will obtain and study fragmented scales whose derived spaces have a similar form. To give just one example, fragmentation of the scale $(\ell_1, \ell_\infty)$ yields the scale $(\ell_2(\ell_1^n), \ell_2(\ell_\infty^n))$, which generates the derived space $\ell_2(Z_2(n))$, where $Z_2(n)$ is the fragmentation of the Kalton-Peck space. And while the Kalton-Peck sequence is strictly singular, the fragmented Kalton-Peck sequence is ``strictly non-singular" (we thank F. Cabello for this name).\medskip

The general theory yields that an admissible couple $(X_0, X_1)$ of Banach spaces for which complex interpolation at $\theta$ yields the space $X_\theta$ generates an exact sequence
\begin{equation*}\label{twix}
\begin{CD}
0@>>> X_\theta@>{j}>> d X_\theta @>{q}>> X_\theta @>>>0\end{CD}
\end{equation*}
The middle space $d X_\theta$ in (\ref{twix}) is called the derived space of the scale $(X_0, X_1)$ at $\theta$. It is especially interesting when $X_\theta = \ell_2$, in which case the space $dX_\theta$ is called a twisted Hilbert space (see below). The exact sequence (\ref{twix}) is said to be {\it trivial} when $j(X_\theta)$ is complemented in
$dX_\theta$. The exact sequence is called {\it singular} when the operator $q$ is strictly singular, which means that its restrictions to infinite
dimensional closed subspaces are never an isomorphism. .\medskip

A drawback in the theory is the scarcity of examples. While it is relatively easy to get $(X, X^*)_{1/2}=\ell_2$, it is rather difficult to calculate the associated derivation and study its properties. The paper \cite{cabe2} presents a complete description of the derivations that appear when considering scales of Lorentz spaces, while the paper \cite{cfg} performs a thorough study of singular derivations. In this paper we continue the previous work by obtaining new examples of derivations, study their properties and show that the hypotheses in the main theorems of \cite{cfg} are actually not necessary.

\section{Exact sequences, twisted sums and centralizers}
\label{sect:twisted-sums}

A {\it twisted sum} of two Banach spaces $Y$ and $Z$ is a Banach space $X$ which has a subspace
isomorphic to $Y$ with the quotient $X/Y$ isomorphic to $Z$. An exact sequence $0 \to Y \to X \to Z \to 0$ of Banach spaces and linear continuous operators is a diagram in which the kernel of each arrow coincides with the image of the preceding one.
By the open mapping theorem this means that the middle space $X$ is a twisted sum
of $Y$ and $Z$.\medskip

A special type of exact sequences appear generated by the complex interpolation method when applied to a pair of spaces as we describe now. A sound background on complex interpolation can be found in \cite{BL,kaltonmontgomery}. Let $\mathbb S$ denote the open strip $\{z\in \mathbb C: 0< {\rm Re}(z)< 1\}$ in the complex
plane, and let $\overline{\mathbb S}$ be its closure. Given an admissible (i.e., a pair that we assume linear and continuously embedded into a Hausdorff topological vector space $W$) pair $(X_0,X_1)$ of complex Banach spaces, let $\Sigma = X_0+X_1$ endowed with the norm $\|x\|=\inf \{\|x_0\|_0 + \|x_1\|_1: x= x_0 + x_1\}$. We denote by $\cl{F}=\cl F(X_0,X_1)$
the space of functions $g:\overline{\mathbb S}\to \Sigma$ satisfying the following conditions:
\begin{enumerate}
\item $g$ is $\|\cdot\|_\Sigma$-bounded and $\|\cdot\|_\Sigma$-continuous on $\overline{\mathbb S}$,
and $\|\cdot\|_\Sigma$-analytic on $\mathbb S$;

\item $g(it)\in X_0$ for each $t\in\R$, and the map $t\in\R\mapsto g(it)\in X_0$ is bounded
and continuous;

\item $g(it+1)\in X_1$ for each $t\in\R$, and the map $t\in\R\mapsto g(it+1)\in X_1$ is bounded
and continuous.
\end{enumerate}

The space $\cl{F}$ is a Banach space under the norm
$\|g\|_\cl{F} = \sup\{\|g(j+it)\|_j: j=0,1; t\in\R \}$.
For $\theta\in[0,1]$, define the interpolation space
$$
X_\theta=(X_0,X_1)_\theta=\{x\in\Sigma: x=g(\theta) \text{ for some } g\in\cl F\}
$$
with the norm $\|x\|_\theta=\inf\{\|g\|_\cl{F}: x=g(\theta)\}$.
So, if $\delta_\theta: \mathcal F\to \Sigma$ denotes the obvious evaluation map $\delta_\theta(f)=f(\theta)$ then $(X_0,X_1)_\theta$ is the quotient of $\cl{F}$ by $\ker\delta_\theta$, and thus it is a Banach space.  For  $0<\theta<1$, we will consider the maps
$\delta_\theta: \mathcal F \to \Sigma$ (evaluation of the function at $\theta$) and
$\delta_\theta': \mathcal F\to \Sigma$ (evaluation of the derivative at $\theta$). Let $B: X_\theta \to \mathcal F$ be a bounded homogeneous selection for $\delta_\theta$ and set
$$d_{\delta_\theta' B}X_\theta = \{(y,z)\in \Sigma  \times X_\theta\, :\,
 y - \delta_\theta' B z\in X_\theta\}
$$
endowed with the quasi-norm
$\|(y,z)\| = \|y -\delta_\theta' B z\|_{X_\theta} +\|z\|_{X_\theta}$.
Under rather general conditions it can be shown that $\|\cdot\|$ is equivalent to a norm. The space $d_{\delta_\theta' B}X_\theta$ is a twisted sum of $X_\theta$ since the embedding $y\to (y,0)$ and quotient map $(y,z) \to z$ yield an exact sequence
$$
\begin{CD}
0@>>> X_\theta @>>> d_{\delta_\theta' B}X_\theta @>>> X_\theta@>>>0
\end{CD}
$$ The map $\Omega_\theta = \delta_\theta' B$ is called the associated derivation. Two different homogeneous bounded selectors $B$ and $V$ yield different derivations although their difference is a bounded map
$\delta_\theta' B - \delta_\theta' V: X_\theta \to X_\theta$, and consequently the spaces
$d_{\delta_\theta' B}X_\theta$ and $d_{\delta_\theta' V}X_\theta$ are isomorphic. When no confusion arises we will call this space $dX_\theta$.\medskip

The key example on which the theory is modeled is the Kalton-Peck twisted Hilbert
space $Z_2$ obtained in \cite{kaltpeck}, which provides the first example of a singular
sequence
$$
\begin{CD}
0@>>> \ell_2@>>> Z_2 @>>> \ell_2 @>>>0.
\end{CD}
$$The space $Z_2$ is the derived space $d\ell_2$ with respect to the scale $(\ell_1, \ell_\infty)$ at $1/2$. The map $\Omega_\theta$ is a homogeneous function $\Omega_\theta: X_\theta\to X_0+ X_1$ with the additional property that there is a constant $C>0$ such that $\|\Omega_\theta(x+y) - \Omega_\theta(x) - \Omega_\theta(y)\|\leq C(\|x\|+\|y\|)$ for all $x,y\in X_\theta$. Such a map is usually called the derivation at $\theta$ (with respect to the given interpolation scale), and can be used to completely describe the derived space and the induced exact sequence. Moreover, an outstanding theorem of Kalton \cite{kaltdiff} establishes a very tight connection between interpolation scales and derivations while one maintains inside the domain of K\"othe spaces, that can be described as follows: Let $X$ be a K\"othe function space on a measure space $M$. A homogeneous map $\Omega$ on $X$ is said to be a centralizer if there is a constant $C>0$ such that $\|\Omega(\xi x) - \xi\Omega(x) \|\leq C\|\xi\|\|x\|$ for all $x\in X$ and $\xi\in L_\infty(M)$. Derivations on K\"othe spaces are centralizers. And Kalton's theorem asserts that given a centralizer $\Omega$ on $X$ satisfying a few technical requirements, there is a couple of K\"othe spaces $X_0, X_1$ such that $X=(X_0, X_1)_{1/2}$ and $\Omega - \Omega_{1/2}$ is bounded. This means that, in practice, one can replace $\Omega$ by $\Omega_{1/2}$.

An example is in order: fix the couple $(\ell_{p_0}, \ell_{p_1})$ and set $\frac{1}{p}= \frac{1-\theta}{p_1}+ \frac{\theta}{p_0}$. The map $B(x)(z) = x|x|^{(\frac{p}{p_1} - \frac{p}{p_0})(z-\theta)}$
is a homogeneous bounded selection for the evaluation map $\delta_{\theta}: \mathcal F\to \ell_{p}$ that yields the derivation $
(\frac{p}{p_1} - \frac{p}{p_0})x\log \frac{|x|}{\|x\|} = (\frac{p}{p_0} - \frac{p}{p_1})\mathscr K(x)  $, where $\mathscr K$ is the so-called Kalton-Peck map $\mathscr K(x)= x\log \frac{\|x\|}{|x|}$ for $x\neq 0$, and $0$ otherwise, that we assume defined only on finitely supported sequences.

\section{Fragmentation, reiteration, interpolation and derivation}
It is a well-known standard fact that when working with K\"othe spaces, in particular with Banach spaces with an unconditional basis, then complex interpolation becomes plain factorization (see \cite[Thm.4.6]{kaltonmontgomery}). In particular, the Lozanovskii decomposition formula allows us to show (see \cite[Theorem 4.6]{kaltonmontgomery}) that the complex interpolation space $X_\theta$ relative to a couple $(X_0, X_1)$ of K\"othe spaces (one of them with the Radon-Nikodym property) on the same measure space is isometric to the space $X_0^{1-\theta} X_1^\theta$,
with
$$
\|x\|_\theta=\inf \{\|y\|_0^{1-\theta}\|z\|_1^{\theta}: y\in X_0, z\in X_1,
|x|=|y|^{1-\theta}|z|^{\theta}\}.
$$
Reasoning as in \cite[p.1165]{kaltonmontgomery} (see also \cite[Section 3.3]{cfg}) one gets that if $a_0(x),a_1(x)$ is an almost optimal Lozanovskii decomposition for $x$ then \begin{equation}\label{kpgeneralized}
\Omega_\theta(x)  = x\, \log \frac{|a_1(x)|}{|a_0(x)|}.
\end{equation}

So let $L$ be a Banach space with an unconditional shrinking basis, so that it is a common unconditional basis for $L$ and $L^*$. Given a finite subset $A\subset \N$, we define the finite dimensional space $L(A)$ as the subspace of $L$ formed by all elements with support contained in $A$ endowed with the norm  $\|x\|_{L(A)}=\|1_Ax\|_L$. One has:

\begin{lemma}\label{frag}  $\left(L(A), {L^*(A)}\right)_\theta  = {(L, L^*)_\theta}(A)$ with derivation $\Omega_A(x)= 1_A\Omega(1_Ax)$.\end{lemma}
\begin{proof} Since $(L, L^*)_\theta = L^{1-\theta}{L^*}^{\theta}$ one immediately gets
$(L(A), {L^*}(A))_\theta = L(A) ^{1-\theta}{{L^*}(A)}^{\theta}$ with equality of norms
\begin{eqnarray*}
\|x\|_{(L(A), {L(A)}^*)_\theta}&=&\inf\lbrace\|y\|_{L(A)}^{1-\theta}\|z\|_{L(A)^*}^\theta\rbrace\\
&=&\inf\lbrace\|1_Ay\|_{L}^{1-\theta}\|1_Az\|_{L^*}^\theta\rbrace\\
&=&\|1_Ax\|_{(L, L^*)_\theta}\\ &=&\|x\|_{{(L, L^*)_\theta}(A)}.
\end{eqnarray*}

To calculate the associated derivation, observe that if $B$ is a homogeneous bounded selection for  $\delta_\theta:\mathcal{F}(L,L^*)\rightarrow(L,L^*)_\theta$ then $B_A(x)=1_AB(1_Ax)$ is a bounded homogeneous selection of $\delta_\theta:\mathcal{F}(L(A),L^*(A))\rightarrow(L,L^*)_\theta(A)$. Thus, the associated derivation is now $\Omega_A(x)={\delta'}_\theta B_A(x)=1_A{\delta'}_\theta B(1_Ax)=1_A\Omega(1_Ax)$.\end{proof}

To see what occurs when the pieces $L(A_n)$ are pasted together using a Banach sequence space $\lambda$ we have to examine the behaviour of vector sums of Banach spaces under interpolation and derivation. Let $\lambda$ be a Banach space with a $1$-unconditional basis $(e_n)$. Given a  Banach space $X$ one can form the Banach space
$$\lambda(X) = \left\{ (x_n) \in X^\N:     \|(x_n)\| = \left \| \sum \|x_n\| e_n \right \|_\lambda <+\infty\right\}.$$
It is part of the folklore that if, moreover, $X$ has an unconditional basis then $\lambda(X)$ also has an unconditional basis. This is however false when one considers an arbitrary sequence $(X_n)$ of Banach spaces since it is well-known that the space $\ell_2(L_{p_n})$ has no unconditional basis when $\lim p_n=1$ \cite[p.27]{lindtzaf}. It is well-known \cite[Section 5.1]{BL} that $(\ell_{p_0}(X_0), \ell_{p_1}(X_1))_\theta = \ell_{p}((X_0, X_1)_\theta)$ for $\frac{1}{p} =\frac{1-\theta}{p_0} + \frac{\theta}{p_1}$. Our purpose now is to obtain a generalized form of this result and calculate the associated derivation.\medskip

Calder\'on's paper \cite{calde} contains a rather general interpolation result for vector sums that we describe now. Let $\Lambda$ be a K\"othe space defined on a measure space $M$. Given a Banach space $X$ one can form the vector valued space Banach $\Lambda(X)$ of measurable functions $f: M\to X$ such that the function $\widehat f (\cdot) = \|f( \cdot)\|_X: M\to \R$ given by $t \to \|f(t)\|_X$ is in $\Lambda$, endowed with the norm $\| \|f(\cdot)\|_X \|_\Lambda$. 

\begin{theorem}\label{Cinterpolation} Fix $0<\theta<1$. Let $(\lambda_0, \lambda_1)$ an interpolation couple of Banach lattices on the same measure space for which
$(\lambda_0, \lambda_1)_\theta = \lambda_0^{1-\theta}\lambda_1^{\theta}$ with associated derivation $\omega_\theta$. Let $(X_0, X_1)$ be an interpolation couple of Banach spaces with associated derivation $\Omega_\theta$ at $\theta$. Assume that $\lambda_0(X_0)$ is reflexive. Then $$(\lambda_0(X_0), \lambda_1(X_1))_\theta = \lambda_0^{1-\theta}\lambda_1^\theta \left ((X_0, X_1)_\theta\right)$$
with associated derivation $\Phi_\theta$ defined on the dense subspace of simple functions as follows: given $f = \sum_{n=1}^N a_n 1_{A_n}$ then
$$\Phi_\theta \left (f \right )= \omega_\theta\left ( \widehat f(\cdot)\right ) \sum_{n=1}^N \frac{a_n}{\|a_n\|}1_{A_n}  + \sum_{n=1}^N  \Omega_\theta (a_n) 1_{A_n}.$$
\end{theorem}

\begin{proof} The identification of the interpolation space is Calder\'on's vector-valued interpolation formula \cite[13.6 (i),(ii)]{calde}. To obtain the derivation we must go to Calder\'on's proof \cite[33.6]{calde}. Let $B(\nu)$ and $B(x)$ be extremals for $\nu \in (\lambda_0, \lambda_1)_\theta$ and $x\in (X_0, X_1)_\theta$, respectively, so that $B(\nu)'(\theta)=\omega_\theta(\nu)$ and $B(x)'(\theta)=\Omega_\theta(x)$. Then formula (5) in \cite{calde} establishes that an extremal for $f\in (\lambda_0(X_0), \lambda_1(X_1))_\theta$ is given by
$$Fx = B(\widehat f)\sum B \left (\frac{a_n}{\|a_n\|}\right) 1_{A_n}.$$
Therefore
\begin{eqnarray*}
\Phi_\theta(f) &=& (Fx)'(\theta)\\
&=& B(\widehat f)'(\theta) \sum B \left (\frac{a_n}{\|a_n\|}\right)(\theta) 1_{A_n} + B(\widehat f)(\theta) \sum B \left (\frac{a_n}{\|a_n\|}\right)'(\theta) 1_{A_n}\\
&=& \omega(\widehat f)\sum \frac{a_n}{\|a_n\|}1_{A_n} + \widehat f \sum \Omega_\theta \left (\frac{a_n}{\|a_n\|}\right)1_{A_n}\\
&=& \omega(\widehat f)\sum \frac{a_n}{\|a_n\|}1_{A_n} + \sum \Omega_\theta (a_n)1_{A_n}
\end{eqnarray*}
by the homogeneity of $\Omega_\theta$.\end{proof}

A particularly interesting case occurs when the spaces $\lambda_j$, $j=0,1$ are $p_j$-convexifications of the same Banach space $\lambda$ with unconditional basis. Precisely, recall that given a Banach space $\lambda$ with a 1-unconditional basis $(e_n)$ and given $1\leq p<+\infty$, its $p$-convexification $\lambda_p$ is defined as the completion of the space of finitely supported sequences endowed with the norm
$$\left \|  \sum_{n=1}^\infty\lambda_n e_n  \right \|_{\lambda_p} = \left \|  \sum_{n=1}^\infty |\lambda_n|^p e_n  \right \|_{\lambda}^{1/p}.$$
We need an improved version of the interpolation inequality:

\begin{lemma}\label{cuenta} Let $\lambda$ be a Banach space with a K-unconditional basis $(e_n)$ and $\frac{1}{p}=\frac{1-\theta}{p_0}+\frac{\theta}{p_1}$. One has
$$\left \| \sum |a_n|^{1-\theta} |b_n|^{\theta}e_n\right \|_{\lambda_p} \leq K^{1/p} \left \| \sum |a_n|e_n\right \|_{\lambda_{p_0}}^{1-\theta} \left \| \sum |b_n|e_n\right \|_{\lambda_{p_1}}^{\theta}.$$
\end{lemma}

The proof is rather straightforward. With this we can obtain the following version of Theorem \ref{Cinterpolation} without the additional hypothesis
that $\lambda_0(X_0)$ has to be reflexive.

\begin{theorem}\label{interpolation} Let $\lambda$ be a Banach space with a 1-unconditional basis, let $0<\theta<1$ and let $\frac{1}{p}=\frac{1-\theta}{p_0}+\frac{\theta}{p_1}$. If $(X_0, X_1)$ is an interpolation couple of Banach spaces with associated derivation $\Omega_\theta$ at $\theta$ then $$(\lambda_{p_0}(X_0), \lambda_{p_1}(X_1))_{\theta}= \lambda_p(X_{\theta})$$ with equality of norms and associated derivation $\Phi_\theta$ defined on finitely supported elements $a= \sum_{n=1}^N a_n e_n\in \lambda_p(X_\theta)$ as
$$\Phi_\theta \left (a \right )= \left(\frac{p}{p_1}-\frac{p}{p_0} \right)\sum_{n=1}^N a_n \log\frac{\|a_n\|}{\|a\|} e_n + \sum_{n=1}^N \Omega_{\theta} (a_n)e_n.$$
\end{theorem}

\begin{proof} The interpolation part is rather standard using the inequality above. What is important for us is to observe that given $a=\sum_{n=1}^N a_ne_n$ with $a_n\in X_{\theta}$ then picking extremals (with respect to the couple $(X_0, X_1)$) $g_n$ such that $g_n(\theta)=a_n$ and $\|g_n\|\leq (1+\varepsilon)\|a_n\|_{X_{\theta}}$ one obtains the extremal (with respect to the
couple $(\lambda_{p_0}(X_0), \lambda_{p_1}(X_1))$)

$$f_n(z)=g_n(z)\left ( \frac{\|a_n\|_{X_{\theta}}}{\|a\|_{\lambda_p(X_{\theta})}}  \right)^{p\left( \frac{1}{p_0}-\frac{1}{p_1} \right)(\theta-z)}.$$

The associated derivation is therefore
$$\Phi_\theta (a) =f'(\theta)= \left(\frac{p}{p_1}-\frac{p}{p_0} \right)\sum_{n=1}^N g_n(\theta) \log \left(\frac{\|a_n\|_{X_{\theta}}}{\|a\|_{\lambda_p(X_{\theta})}} \right)e_n+ \sum_{n=1}^N g'_n(\theta)e_n.$$

Since there is no loss of generality in assuming that $g_n'(\theta)=\Omega_{\theta}(a_n)$, we are done.\end{proof}

The derivation obtained in Theorem \ref{Cinterpolation} matches that obtained in Theorem \ref{interpolation}: in the particular case in which $\lambda_0, \lambda_1$ are Banach spaces with unconditional basis and we set $1_{A_n}=e_n$; if, moreover, $\lambda_{p_0}$ (resp. $\lambda_{p_1}$) is the $p_0$ (resp. $p_1$) -convexification of $\lambda$ then $\omega$ is the ``vectorial Kalton-Peck" map $\left(\frac{p}{p_1}-\frac{p}{p_0} \right)\sum_{n=1}^N a_n \log\frac{\|a_n\|}{\|a\|} e_n$ which yields
\begin{eqnarray*} \Phi_\theta \left (f \right ) &=& \omega_\theta\left ( \widehat f(\cdot)\right ) \sum_{n=1} ^N \frac{a_n}{\|a_n\|}e_n  + \sum_{n=1}^N  \Omega_\theta (a_n) e_n\\
&=&\left(\frac{p}{p_1}-\frac{p}{p_0} \right)\sum_{n=1}^N \|a_n\| \log\frac{\|a_n\|}{\|a\|} e_n  \sum_{n=1}^N \frac{a_n}{\|a_n\|}e_n + \sum_{n=1}^N  \Omega_\theta (a_n) 1_{A_n}\\
&=&\left(\frac{p}{p_1}-\frac{p}{p_0} \right)\sum_{n=1}^N a_n \log\frac{\|a_n\|}{\|a\|} e_n +  \sum_{n=1}^N  \Omega_\theta (a_n) 1_{A_n}.
\end{eqnarray*}

When $\lambda$ is fixed and $(X_0, X_1)$ is an interpolation couple with associated derivation $\Omega_\theta$ at $\theta$ it comes as no surprise that the derivation associated to the scale $(\lambda(X_0), \lambda(X_1))$ is better described as
 $\lambda(\Omega_\theta)$, with the precise meaning we give it now. If $W$ is the ambient space associated to the couple $(X_0, X_1)$ then there is no loss of generality
by assuming that the ambient space associated to the couple $(X_0^N,X_1^N)$ is $W^N$ and the ambient space for $(\lambda(X_0), \lambda(X_1))$ is $W^\N$. Thus, the derivation associated to $(X_0^N,X_1^N)_\theta = {(X_0, X_1)_\theta}^N$ is ${\Omega_\theta}^N$ and therefore the derivation associated to $(\lambda(X_0), \lambda(X_1))_\theta = \lambda( (X_0, X_1)_\theta)$ is
$\lambda(\Omega_\theta)$ with the meaning that for every finitely supported $x= (x_1, \dots, x_N, 0,\dots) \in \lambda((X_0, X_1)_{\theta}) = (\lambda(X_0), \lambda (X_1))_{\theta}$ one has
$$\lambda(\Omega_\theta)(x) = (\Omega_\theta(x_1), \dots, \Omega_\theta(x_N), 0, \dots).$$

The role of $\lambda$ here is that whilst $\lambda(\Omega_\theta)$ takes values in $W^\N$, for every $x, y \in \lambda((X_0, X_1)_{\theta}) = (\lambda(X_0), \lambda (X_1))_{\theta}$ the Cauchy differences
$\lambda(\Omega_\theta)(x+y) - \lambda(\Omega_\theta)(x)  - \lambda(\Omega_\theta)(y)$ lie in $\lambda((X_0, X_1)_{\theta})$ and moreover
$\|\lambda(\Omega_\theta)(x+y) - \lambda(\Omega_\theta)(x)  - \lambda(\Omega_\theta)(y)\| \leq C (\|x\| + \|y\|)$.\medskip

Let us present a few tangible examples:\begin{itemize}
\item  Pick the couples $(\ell_p, \ell_{p^*})$ and $(\ell_{p^*}, \ell_p)$ (in reversed order) and let us calculate the derivation at $\ell_2(\ell_2)$. In the first case, the derivation at $1/2$ is $
\mathcal{K}(x)=\left(\frac{2}{p}-\frac{2}{p^*}\right)\sum_kx_k\log\dfrac{|x_k|}{\|x\|}u_k$
where $(u_k)$ denotes the canonical basis of $\ell_2$; in the second case, the derivation at $1/2$ is
$-\mathcal{K}(x)=\left(\frac{2}{p}-\frac{2}{p^*}\right)\sum_kx_k\log\dfrac{|x_k|}{\|x\|}e_k$. Thus, according to Theorem \ref{Cinterpolation}
interpolation between $\ell_p(\ell_{p*})$ and $\ell_{p^*}(\ell_p)$ at $1/2$ yields $\ell_2(\ell_2)$ with associated derivation at $a=\sum_{k=1}^Na_ku_k$ with $a_k\in \ell_2$ given by
$$
\Phi(a)=\left(\dfrac{2}{p}-\dfrac{2}{p^*}\right)\sum_{k=1}^N\left(a_k\log\dfrac{\|a_k\|}{\|a\|}-\sum_na_{k}(n)\log\dfrac{|a_k(n)|}{\|a_k\|}e_n\right)u_k.
$$
\item Pick the sequence of finite dimensional couples $(\ell_{p_n}^{k_n}, \ell_{p_n^*}^{k_n})$ and let us calculate the derived space at $\ell_2(\ell_2^{k_n})$. With all previously mentioned reservations the derivation can be written as
    $$\left( \left (\sum_{j=1}^{k_n} x_j^nu_j^n\right)_n\right) \longrightarrow \ell_2\left( \left(\frac{2}{p_n}-\frac{2}{p_n^*}\right)\sum_{j=1}^{k_n} x_j^n\log\dfrac{|x_j^n|}{\|\sum_{j=1}^{k_n} x_j^nu_j^n\|_2}u_j^n\right).$$
\item According to \cite{john} and  \cite[p.21]{pisier}, when $\lim p_n=2$ and $k_n\to \infty$  are adequately chosen the space $\ell_2\left( \ell_{p_n}^{k_n}\right)$ is asymptotically Hilbert and non-Hilbert. Proceeding as in \cite{suarez} one can show that the derived space is a twisted Hilbert asymptotically Hilbert space that is not Hilbert.

\item Let $(A_n)$ be a partition of $\N$, $\lambda$ a Banach space with an $1$-unconditional basis and $L$ a Banach space with a shrinking unconditional basis. Suppose that $\Omega_{A_n, \theta}$ is the centralizer associated to the scale $(L(A_n),L^*(A_n))$ at $\theta$. It follows as in Proposition \ref{interpolation} that $\big(\lambda(L(A_n)),\lambda(L^*(A_n))\big)_\theta=\lambda\big((L(A_n),L^*(A_n))_\theta\big)=\lambda\big((L,L^*)_\theta (A_n)\big)$ with associated centralizer is $\lambda(\Omega_{A_n, \theta})$ with the meaning that for every finitely supported $x= (x_1, \dots, x_N, 0,\dots) \in \lambda\big((L,L^*)_\theta (A_n)\big)$ one has $\lambda(\Omega_{A_n, \theta})(x) = (\Omega_{A_n, \theta}(x_n)).$
\item According to \cite[Prop. 3.6]{cfg}, $\lambda_p = (\lambda, \ell_\infty)_{1/p}$ with associated derivation $p\; \mathscr K$. Therefore, by the reiteration theorem \cite[Section 4.6]{BL}, one has
$$(\lambda_{p_0}, \lambda_{p_1})_\theta = ((\lambda, \ell_\infty)_{1/p_0}, (\lambda, \ell_\infty)_{1/p_1})_\theta = (\lambda, \ell_\infty)_p$$
when $\frac{1}{p} = \frac{1-\theta}{p_0} + \frac{\theta}{p_1}$ with associated derivation $\left (\frac{p}{p_1}-\frac{p}{p_0} \right )\mathscr K$ as it is calculated in \cite[Prop. 3.7]{cfg}. In general, given $(X_0, X_1)$ a compatible couple and one sets $X_{\theta}=(X_0,X_1)_{\theta}$ then the reiteration theorem claims that $(X_{\theta_0}, X_{\theta_1})_{\eta}= X_{\theta}$ holds with equal norms where $\theta=(1-\eta)\theta_{0}+\eta\theta_1.$ Let us denote as usual by $\Omega_{\theta}$ the derivation corresponding to $X_{\theta}=(X_0,X_1)_{\theta}$. The associated derivation to $(X_{\theta_0}, X_{\theta_1})_{\eta}$  is $(\theta_1-\theta_0)\Omega_{\theta}$ (see \cite[Prop.2.3]{cfg} for the case of K\"othe spaces and  \cite{ccfg} for a more general form for this iterated derivation).
\end{itemize}

\section{Singularity properties}
On the opposite side of trivial exact sequences one encounters singular sequences which, as we have already said at the Introduction, are those in which the quotient map is a strictly singular operator. Thus, if one defines singular quasi-linear map \cite{ccs,strict} as one whose restrictions to every infinite dimensional closed subspace are never trivial then one gets that an exact sequence is singular if and only if it is induced by a singular quasi-linear map \smallskip

For every $0\leq p<+\infty$ the Kalton-Peck map $\mathscr K: \ell_p\to\ell_p$ is
singular. The proof for $p>1$ is in \cite{kaltpeck}, the proof for $p=1$ is in \cite{strict} and a proof valid for all $p<+\infty$ can be found in \cite{ccs}.

\adef A quasi-linear map $\Omega: Z\to Y$ will be called \emph{strictly non-singular} if every infinite dimensional subspace $A\subset Z$ contains an infinite dimensional subspace $B\subset A$ so that $\Omega_{|B}$ is trivial.
\zdef
Contrarily to what occurs with strict singularity, strict non-singularity does not have a straightforward translation to the operator language since a quotient map $q: X\to Z$ such that every subspace of $X$ contains a further subspace on which $q$ becomes an isomorphism is itself an isomorphism. An exact sequence
$0\to Y \to X \to Z \to 0$ will be called strictly non-singular if its associated quasi-linear map is strictly non-singular. A quotient map $q: X \to Z$ will be called strictly non-singular if the associated sequence $0\to \ker q \to X \to Z \to 0$ is strictly non-singular. One can easily prove that a quotient map $q:X\to Z$ is strictly non-singular if and only if for every infinite dimensional subspace $A\subset Z$ the exact sequence $0\to Y\to  q^{-1}(A) \to A \to  0$ is not singular.\smallskip

Natural examples of strictly non-singular derivations will be given soon. A sequence $(x_n)$ in $X$ is called weakly-$p$-summable, $1<p<+\infty$, if $(x^*(x_n))_n\in \ell_p$ for every $x^*\in X^*$; equivalently, if $\sup_{\|(\theta_n)\|_{p^*}\leq 1} \|\sum \theta_n x_n\|<+\infty$ (see, e.g.,  \cite{diest}). We need from \cite{wp1,wp2,wp3} the notion of property $\mathcal W_p$:

\adef  A Banach space is said to have property  $\mathcal W_p$ if it is reflexive and every weakly null sequence admits a weakly $p$-summable subsequence.
\zdef

Finally, recall that a Banach space is said to be $\ell_p$-saturated if every infinite dimensional closed subspace contains a subspace isomorphic to $\ell_p$. One has
\begin{lemma}\label{oldie} An exact sequence $0\to Y \to X \to Z \to 0$ in which $Z$ is $\ell_p$-saturated and $X$ has the $\mathcal W_{p^*}$ property is strictly non-singular.
\end{lemma}

\begin{proof} Let $H$ be an infinite dimensional subspace of $Z$ and let $(h_n)$ a sequence in $H$ equivalent to the canonical basis of $\ell_p$ inside $H$. Since $q$ is open, there exists a constant $C>0$ such that for every $h_n$ we can choose $x_n\in X$ with $\|x_n\|\leq C$ and such that $qx_n=h_n$. Since $X$ is reflexive we can suppose that $(x_n)$ is weakly convergent to, say, $x$. Thus $qx=0$. By the $\mathcal W_{p^*}$ property of $X$, there exists a weakly $p^*$-summable subsequence $(x_k-x)$. Thus, the linear application $h_k\to x_k -x $ is a continuous selection for $q_{|[x_k]}$.
\end{proof}
Examples of spaces with property $\mathcal W_p$ are provided by \cite[Theorem 1]{cscam}: if $\lambda$ is a Banach space with unconditional basis with property $\mathcal W_p$ and $X$ is a Banach space with property $\mathcal W_p$ then also $\lambda(X)$ has property $\mathcal W_p$. According to \cite[Remark 3]{cscam} the result is false for
an arbitrary $\ell_p$-sum of a sequence of spaces with property $\mathcal W_{p^*}$, although it still works for sequences of finite dimensional spaces. The argument we will need is essentially contained in the proof of \cite[Theorem 1]{cscam} (see also \cite{OS} for a more general result).

\begin{cor}\label{2sum} Given a sequence $(F_n)$ of finite-dimensional spaces, $\ell_p(F_n)$ has the $\mathcal W_{p^*}$ property.
\end{cor}

The paper \cite{cfg} studied the nontriviality and singularity of the sequences $\Omega_\theta$ in terms of the initial couple $(X_0, X_1)$. More precisely, let  $X$ be a Banach space with a $1$-unconditional basis. Following \cite{cfg}, we consider the parameter
$$
A_{X}(n)=\sup \{\|x_1+\ldots+x_n\|: \|x_i\|\leq 1,\; n<x_1< \ldots < x_n\}.
$$
Given two real functions $f,g$ we will write $f\sim g$ to mean that $0<\lim \inf f(t)/g(t) \leq \lim \sup f(t)/g(t)<+\infty$. One then has

\begin{theorem} \cite[Proposition 5.7]{cfg} \label{logb}
Let $(X_0, X_1)$ be an interpolation couple of Banach spaces with a common $1$-unconditional
basis, and let $0<\theta<1$. If
\begin{enumerate}
\item $A_{X_0} \not\sim A_{X_1}$,
\item $A_{X_0}^{1-\theta} A_{X_1}^\theta \sim A_{X_{\theta}}$,
\item $A_{X_{\theta}} \sim A_Y$ for all infinite dimensional subspaces $Y\subset X_\theta$,
\end{enumerate} then $\Omega_\theta$ is singular.
\end{theorem}
The paramount example is provided by the scale $(\ell_1, \ell_\infty)$, which yields at $\theta=1/p$ the interpolation space $\ell_{p}$ and induces the Kalton-Peck sequence $\mathscr K$ and derived space $Z_p$. Since $A_{\ell_p}(n)=n^{1/p}$ and conditions (1), (2), (3) are verified, $\mathscr K$ is singular for all $1<p<\infty$. We study now to what extent the conditions are necessary to get $\Omega_\theta$ nontrivial or singular.

\adef
Let $(X_0, X_1)$ be an interpolation couple of Banach spaces with a common $1$-unconditional
basis, and let $0<\theta<1$. We will say that the spaces $(X_0, X_1)$ are $A$-different if $A_{X_0} \not\sim A_{X_1}$; we will say that they
$A$-interpolate at $\theta$ if $A_{X_0}^{1-\theta} A_{X_1}^\theta \sim A_{X_{\theta}}$; and we will say that they are homogeneous at $\theta$ if $A_{X_{\theta}} \sim A_Y$ for all infinite dimensional subspaces $Y\subset X_\theta$.\zdef

With some abuse of notation we will say that the scale has those properties when it is not necessary to specify $X_0, X_1$ and $\theta$. For instance, whenever $A_{X_0}\sim A_{X_1}$ then the scale $A$-interpolates

\subsection{The scale of weighted $\ell_p$ spaces} Let $\omega=(\omega_i)_{i=1}^\infty$ be a weight; i.e., a sequence of strictly positive real numbers. Consider the weighted $\ell_p$ spaces defined by
$$
\ell_p(\omega)=\left\lbrace(y_i)_{i=1}^\infty\in\mathbb C^{\N}:\sum |y_i|^p\omega_i<\infty\right\rbrace
$$
with norm $
\|y\|=\left(\sum |y_i|^p\omega_i\right)^{1/p} $. Let $\omega_0, \omega_1$ be weights and pick the scale $\left(\ell_p(\omega_0), \ell_p(\omega_1)\right)$. As it is well known (see \cite{BL} section 5.4) $\left(\ell_p(\omega_0), \ell_p(\omega_1) \right)_\theta = \ell_p(\omega_0^{1-\theta}\omega_1^\theta)$. Since the map $Bx(z)=x\left(\dfrac{\omega_1}{\omega_0}\right)^{(\theta-z)/p}$ is a homogeneous bounded selection for the evaluation map $\delta_\theta:\mathcal F\rightarrow\ell_p(\omega_0^{1-\theta}\omega_1^\theta)$ the associated derivation is the linear map $x\to -\frac{1}{p}x\log \frac{\omega_1}{\omega_0}$, hence trivial. In particular,
$\left(\ell_2(\omega^{-1}), \ell_2(\omega) \right)_{1/2}=\ell_2$ with trivial derivation. On the other hand
\begin{lemma} $A_{\ell_p(\omega)}(n) = n^{1/p}$.
\end{lemma}
\begin{proof} We know that $A_{\ell_p}(n)=n^1/p$.
If $x\in\ell_p(\omega)$ then $x\omega^{1/p}\in\ell_p$ and $\|x\omega^{1/p}\|_p=\|x\|_{\ell_p(\omega)}$ so the parameter $A_{\ell_p(\omega)}(n) \leq n^{1/p}$.
Now given $x\in\ell_p$, then $\omega^{-1/p}x\in\ell_p(\omega)$ and $\left\|\omega^{-1/p}x \right\|_{\ell_p(\omega)}=\|x\|_p$, so choosing the elements $\left\lbrace \omega^{-1/p}e_j\right\rbrace_{j=1}^n$ we obtain that the parameter $A_{\ell_p(\omega)}(n) = n^{1/p}$.
\end{proof}

Thus, the scale of weighted $\ell_p(\omega)$ spaces are not $A$-different, although it $A$-interpolates and is $A$-homogeneous. In fact, one would be easily tempted to believe that scales with equal $A$ --who, therefore $A$-interpolate--- should induce trivial derivations. However, it is not so:

\subsection{The scale of Lorentz sequence spaces} Consider the scales of Lorentz $\ell_{p,q}$ spaces, whose norm comes defined by
$\|x\|_{p,q}=\frac{p}{q}\left(\sum_{n=1}^\infty x^*(n)^q\left(n^{q/p}-(n-1)^{q/p}\right)\right)^{1/q}$
if $q<\infty$, and $\|x\|_{p,q}=\sup n^{1/p}x^*(n)$ for $q=\infty$. Recall from \cite{cabe2} that $(\ell_{p_0, q_0}, \ell_{p_1, q_1})_\theta = \ell_{p,q}$
with derivation
$$
\Omega(x)=  q\left(\dfrac{1}{q_1}-\frac{1}{q_0}\right)\mathscr K(x)    +     \left(\frac{q}{p}\left(\dfrac{1}{q_0}-\frac{1}{q_1}\right)-\left(\dfrac{1}{p_0}-\frac{1}{p_1}\right)\right)\kappa(x)
$$
Here $\kappa$ denotes the Kalton map \cite{cabe2}. It is proved in \cite[Proposition 2]{cabe2} that $\kappa$ is strictly non-singular. On the other hand, it is not hard to check that $A_{\ell_{p,q}}(n)=n^{1/\min\lbrace p, q\rbrace}.$ One thus has that $q_0, q_1 \geq p$ one has $A_{\ell_{p,q_0}} \sim A_{\ell_{p,q_1}}$, the spaces $(\ell_{p,q_0}, \ell_{p,q_1})$  $A$-interpolate and are $A$-homogeneous at every $\theta$. Moreover

\begin{lemma} For $q_0, q_1 \geq p$ the induced derivations are strictly singular.
\end{lemma}
\begin{proof} Indeed, by solving the equation system that appears in the derivation
$$\left\{
  \begin{array}{ll}
    \left(\frac{q}{p}\left(\frac{1}{q_0}-\frac{1}{q_1}\right)-\left(\frac{1}{p_0}-\frac{1}{p_1}\right)\right)=0 & \hbox{;} \\
    p^{-1} = (1-\theta)p_0^{-1} +  \theta p_1^{-1} & \hbox{;} \\
    q^{-1}= (1-\theta)q_0^{-1} + \theta q_1^{-1} & \hbox{}
  \end{array}
\right .$$
one gets that the associated derivation of the scales in which $q_0/p_0 = q_1/p_1$ is the Kalton-Peck map (up to a constant factor); those in which $q_0=q_1$ have the Kalton map as derivation (up to a constant factor); and the associated derivation to every other interpolation scale is a linear combination of both. Since the Kalton-Peck map is strictly singular and the Kalton map is strictly non-singular, all those combinations have strictly singular derivations (see also \cite[Example 1 and Proposition 2]{cabe2}).\end{proof}


It is even possible to obtain $\ell_2$ as interpolated space: the couple $(\ell_{p,p^*}, \ell_{p^*,p})$
has equal $A$ since $A_{\ell_{p, p^*}} = A_{\ell_{p^*, p}}$, does not $A$-interpolate yet it still provides a singular derivation.

\subsection{Fragmented scales} We study now scales obtained by finite-dimensional fragmentation of other scales.

\adef Let $(A_n)$ be a partition of $\N$, $\lambda$ a Banach space with an $1$-unconditional basis and $L$ a Banach space with a shrinking unconditional basis. We shall refer to the interpolation scale obtained from the couple $(\lambda(L(A_n)), \lambda(L^*(A_n)))$ as the $\lambda$-fragmented scale of $L$ according to the partition $(A_n)$ of $\N$.\zdef

Let us consider first the particularly interesting case of the fragmentation of the Kalton-Peck sequence $\mathscr K$.
We already know that by $\ell_2$-fragmentation of $\mathscr K$ we obtain a new derivation $\ell_2({\mathscr K}_{|\ell_2(A_n)})$. The derivation $\mathscr K$ is symmetric, which means, roughly speaking, that its restriction to any finite dimensional $\ell_2(A)$ only depends on the size of the set $A$. And this implies that when $\sup |A_n|< +\infty$ the restrictions $\mathscr {K_2}_{|\ell_2(A_n)}$ are ``uniformly trivial" and therefore $\ell_2({\mathscr K}_{|\ell_2(A_n)})$ is trivial. When $\sup |A_n|=+\infty$ one however has

\begin{prop}\label{fraguel} If $\sup |A_n|=+\infty$ then $\ell_2({\mathscr K}_{|\ell_2(A_n)})$ is not trivial and it is strictly non-singular.
\end{prop}
\begin{proof} The nontriviality  can be deduced from \cite[Theorem 6.3]{kaltpeck} and its strictly non-singular character follows from
Lemma \ref{oldie} and Corollary \ref{2sum}.\end{proof}

\subsection{Scales of (fragmented) weak Hilbert spaces}
We refer the reader to \cite{pisier} for the definition and properties of weak-Hilbert spaces. What we need here is that the Tsirelson $2$-convexified $\mathcal T_2$ space is a weak Hilbert space with unconditional basis. One has:

\begin{prop}\label{weakH}$\;$
 \begin{enumerate}
 \item The couple $(\mathcal T_2, \mathcal T_2^*)$ yields $(\mathcal T_2, \mathcal T_2^*)_{1/2}=\ell_2$ with nontrivial derivation. The scale fails (1) and verifies (2, 3) from Theorem \ref{logb}.
\item Pick the partition of $\N$ given by the sets $A_n = \{2^{n-1}, \dots, 2^{n}-1\}$. The fragmented scale verifies $(\ell_2(\mathcal T_2(A_n)), \ell_2(\mathcal T_2(A_n)^*))_{1/2}=\ell_2$ with trivial derivation.\end{enumerate}
\end{prop}
\begin{proof} Recall that a Banach space is said to have property $(H)$ \cite{pisier} if there is a function $f$ so that any $\lambda$-unconditional finite sequence $(x_1, \dots, x_N)$ verifies an estimate
$$f(\lambda)^{-1}\sqrt{N} \leq \left\|\sum_{n=1}^N x_n\right\| \leq f(\lambda)\sqrt{N}$$
and that weak Hilbert spaces enjoy property $(H)$  \cite{pisier}. Since $\mathcal T_2$ is a space with unconditional basis and property $(H)$ it must therefore verify $A_{\mathcal T_2}(n)\sim \sqrt{n}$ as well as its dual. Therefore,  the couple $(\mathcal T_2, \mathcal T_2^*)$ fails (1). Set now $\theta=1/2$ so that  $(\mathcal T_2, \mathcal T_2^*)_{1/2}$ is a Hilbert space and thus conditions (2) and (3) are obviously verified. The induced derivation $\Omega_{1/2}$ is not trivial since, otherwise, $\mathcal T_2^*$ should be a weighted version of $\mathcal T_2$. A proof for this result in complete generality valid for K\"othe spaces will appear in \cite{cfg}; a proof valid for a couple $(X_0, X_1)$ of spaces with a common unconditional basis appears mentioned without proof in \cite{caskal} and can be done as follows: from \cite[Lemma 1]{ccs} we know that  if the derivation $\Omega_\theta $ is trivial then there is a function $f\in \ell_\infty$ so that $\Omega_\theta (x) -  f x \in X_\theta$ and is bounded there. The rest is simple, just pick  $w_0=e^{- \theta f }$ and $w_1=e^{(1-\theta)f}$ and form the couple $(X_\theta(w_0), X_\theta(w_1))$ that yields
$(X_\theta(w_0), X_\theta(w_1))_\theta = X_\theta(w_0^{1-\theta}w_1^{\theta})= X_\theta$ with derivation $\log (w_0/w_1) x = fx$, obtained from the extremal $w_0^{1-z} w_1^z x$. Since this is  at bounded distance from $\Omega_\theta$, Kalton's uniqueness theorem \cite{kaltdiff} yields that $X_0, X_1$ are, up to equivalent norms, weighted version one of the other.\medskip

Assertion (2) is somewhat trivial because the spaces $\mathcal T_2(A_n)$ are uniformly isomorphic to $\ell_2^{2^ns}$. More precisely, observe that for given $x=\sum \lambda_je_j\in \ell_2(A_n)$ the constant holomorphic function $F_{A_n}(z) = x$ is an extremal since $\|x\|_{\mathcal T_2(A_n)} \leq \|x\|_{\mathcal T_2(A_n)^*} \leq \sqrt{2}\|x\|_2$. Therefore
the derivation is $0$.\end{proof}


\end{document}